\documentclass[12pt]{amsart}
\usepackage{amscd,amsmath,amsthm,amssymb}
\usepackage{amsfonts,amssymb,amscd,amsmath,enumerate,verbatim}
\usepackage[left]{lineno}
\usepackage{pstricks, pst-plot,pst-3d}
\usepackage{tikz}
\newpsstyle{fatline}{linewidth=1.5pt}
\definecolor{verylight}{gray}{0.97}
\definecolor{light}{gray}{0.9}
\definecolor{medium}{gray}{0.85}
\definecolor{dark}{gray}{0.6}
\def\Gc{{\mathcal G}}

\def\Gc{{\mathcal G}}

\def\opn#1#2{\def#1{\operatorname{#2}}} 
\opn\chara{char} \opn\length{\ell} \opn\pd{pd} \opn\rk{rk}
\opn\projdim{proj\,dim} \opn\injdim{inj\,dim} \opn\rank{rank}
\opn\depth{depth} \opn\grade{grade} \opn\height{height}
\opn\embdim{emb\,dim} \opn\codim{codim}
\opn\Cl{Cl}

\opn\Tr{Tr} \opn\bigrank{big\,rank}
\opn\superheight{superheight}\opn\lcm{lcm}
\opn\trdeg{tr\,deg}
	\opn\reg{reg} \opn\lreg{lreg} \opn\ini{in} \opn\lpd{lpd}
	\opn\size{size} \opn\sdepth{sdepth}
	\opn\link{link}\opn\fdepth{fdepth}\opn\lex{lex}
	\opn\tr{tr}\opn\del{del}
	\opn\type{type}
	\opn\gap{gap}
	\opn\arithdeg{arith-deg}
	\opn\revlex{revlex}
	\opn\div{div} \opn\Div{Div} \opn\cl{cl} \opn\Cl{Cl}
	\opn\Spec{Spec} \opn\Supp{Supp} \opn\supp{supp} \opn\Sing{Sing}
	\opn\Ass{Ass} \opn\Min{Min}\opn\Mon{Mon}
	\opn\Ann{Ann} \opn\Rad{Rad} \opn\Soc{Soc}
	\opn\Im{Im} \opn\Ker{Ker} \opn\Coker{Coker} \opn\Am{Am}
	\opn\Hom{Hom} \opn\Tor{Tor} \opn\Ext{Ext} \opn\End{End}
	\opn\Aut{Aut} \opn\id{id}
	
	\opn\nat{nat}
	\opn\pff{pf}
	\opn\Pf{Pf} \opn\GL{GL} \opn\SL{SL} \opn\mod{mod} \opn\ord{ord}
	\opn\Gin{Gin} \opn\Hilb{Hilb}\opn\sort{sort}
	\opn\PF{PF}\opn\Ap{Ap}
	\opn\mult{mult}
	\opn\bight{bight}
	\opn\div{div}
	\opn\Div{Div}
	\opn\aff{aff}
	\opn\relint{relint} \opn\st{st}
	\opn\lk{lk} \opn\cn{cn} \opn\core{core} \opn\vol{vol}  \opn\inp{inp} \opn\nilpot{nilpot}
	\opn\link{link} \opn\star{star}\opn\lex{lex}\opn\set{set}
	\opn\width{wd}
	\opn\Fr{F}
	\opn\QF{QF}
	\opn\G{G}
	\opn\type{type}\opn\res{res}
	\opn\conv{conv}
	\opn\Deg{Deg}
	\opn\Sym{Sym}
	\opn\Con{Con}
	\opn\gr{gr}
	
	\def\pot#1#2{#1[\kern-0.28ex[#2]\kern-0.28ex]}

	\opn\dirlim{\underrightarrow{\lim}}
	\opn\inivlim{\underleftarrow{\lim}}
	\let\to=\rightarrow
	
	\def\Implies{\ifmmode\Longrightarrow \else
		\unskip${}\Longrightarrow{}$\ignorespaces\fi}
	\def\implies{\ifmmode\Rightarrow \else
		\unskip${}\Rightarrow{}$\ignorespaces\fi}
	\def\iff{\ifmmode\Longleftrightarrow \else
		\unskip${}\Longleftrightarrow{}$\ignorespaces\fi}

	\let\:=\colon
	\newtheorem{Theorem}{Theorem}[section]
	\newtheorem{Lemma}[Theorem]{Lemma}

	\newtheorem{Definition}[Theorem]{Definition}
	
	\newtheorem{Conjecture}[Theorem]{Conjecture}

	\let\epsilon\varepsilon
	\let\kappa=\varkappa
	\def\qed{\ifhmode\textqed\fi
		\ifmmode\ifinner\quad\qedsymbol\else\dispqed\fi\fi}
	\def\textqed{\unskip\nobreak\penalty50
		\hskip2em\hbox{}\nobreak\hfil\qedsymbol
		\parfillskip=0pt \finalhyphendemerits=0}
	\def\dispqed{\rlap{\qquad\qedsymbol}}
	
	\opn\dis{dis}
	\def\pnt{{\raise0.5mm\hbox{\large\bf.}}}
	
	\opn\Lex{Lex}

\begin{document}
\title[Modular lattices]{Modular lattices and algebras with straightening laws}

\author[T.~Hibi]{Takayuki Hibi}
\author[S.~A.~ Seyed Fakhari]{Seyed Amin Seyed Fakhari}

\address{(Takayuki Hibi) Department of Pure and Applied Mathematics, Graduate School of Information Science and Technology, Osaka University, Suita, Osaka 565--0871, Japan}
\email{hibi@math.sci.osaka-u.ac.jp}
\address{(Seyed Amin Seyed Fakhari) Departamento de Matem\'aticas, Universidad de los Andes, Bogot\'a, Colombia}
\email{s.seyedfakhari@uniandes.edu.co}



\begin{abstract}
The conjecture that every modular lattice is integral is disproved.
 \end{abstract}	
\maketitle
\thispagestyle{empty}

\section*{Introduction}
A partially ordered set is called a poset.  Every poset to be considered is finite.  We say that a poset $P$ is {\em integral} if there exists an algebra with straightening laws \cite{Eis} on $P$ over a field $K$ which is a homogeneous domain.  Every distributive lattice is integral \cite{Hibi}.  In the present paper, we disprove the conjecture, proposed in \cite{Hibi}, that every modular lattice is integral.

\section{Algebras with straightening laws}
Let $R = \bigoplus_{n=0}^{\infty} R_n$ be a noetherian graded algebra over a field $R_0=K$.  Let $P$ be a poset and suppose that an injection $\varphi: P \hookrightarrow \bigcup_{n=1}^{\infty} R_n$ for which the $K$-algebra $R$ is generated by $\varphi(P)$ over $K$ is given.  A {\em standard monomial} is a homogeneous element of $R$ of the form $\varphi(\gamma_1) \varphi(\gamma_2)\cdots \varphi(\gamma_s)$, where $\gamma_1 \leq \gamma_2 \leq \cdots \leq \gamma_s$ in $P$.  We call $R$ an {\em algebra with straightening laws} \cite{Eis} on $P$ over $K$ if the following conditions are satisfied:
\begin{itemize}
\item[]
(ASL\,-1)
The set of standard monomials is a $K$-basis of $R$;
\item[]
(ASL\,-2)
If $\alpha$ and $\beta$ in $P$ are incomparable and if
\begin{eqnarray*}
\label{ASL}
\, \, \, \, \, \, \, \, \, \, \, \, \, \, \, \, \, \, \, \,
\varphi(\alpha)\varphi(\beta)
= \sum_{i} r_i\,\varphi(\gamma_{i_1})\varphi(\gamma_{i_2}) \cdots , \, \, \, 0 \neq r_i \in K, \, \, \, \gamma_{i_1}\leq \gamma_{i_2} \leq \cdots,
\end{eqnarray*}
is the unique expression for $\varphi(\alpha)\varphi(\beta) \in R$ as a linear combination of distinct standard monomials guaranteed by (ASL\,-1), then $\gamma_{i_1} \leq \alpha, \beta$ for every $i$.
\end{itemize}
The right-hand side of the relation in (ASL\,-2) is allowed to be the empty sum $(=0)$.  We abbreviate an algebra with straightening laws as ASL.  The relations in (ASL\,-2) are called the {\em straightening relations} for $R$.

Let $S=K[x_{\alpha} : \alpha \in P]$ denote the polynomial ring in $|P|$ variables over $K$ and define the surjective ring homomorphism $\pi: S \to R$ by setting $\pi(x_\alpha) = \varphi(\alpha)$. The defining ideal $I_R$ of $R=\bigoplus_{n=0}^{\infty} R_n$ is the kernel $\ker(\pi)$ of $\pi$. When $\alpha,\beta \in P$ are incomparable, we introduce the polynomial
\[
f_{\alpha,\beta}
:= x_\alpha x_\beta - \sum_{i} r_i\,x_{\gamma_{i_1}}x_{\gamma_{i_2}} \cdots x_{\gamma_{i_{n_i}}}, \, \text{ with } 0 \neq r_i \in K, \, \, \, \gamma_{i_1}\leq \gamma_{i_2} \leq \cdots,
\]
arising from (ASL-2). Then $f_{\alpha,\beta} \in I_R$. Let $\Gc_R$ denote the set of those polynomials $f_{\alpha,\beta}$ for which  $\alpha$ and $\beta$ are incomparable in $P$. Let $<_{\mathrm{rev}}$ denote the reverse lexicographic order \cite[Example~2.1.2 (b)]{HHgtm260} on $S$ induced by an ordering of the variables for which $x_{\alpha} <_{\mathrm{rev}} x_{\beta}$ if $\alpha < \beta$ in $P$. It follows from (ASL-1) that $\Gc_R$ is a Gr\"obner basis of $I_R$ with respect to $<_{\mathrm{rev}}$. In particular, $I_R$ is generated by $\Gc_R$.

Recall that a {\em homogeneous algebra} is a noetherian graded algebra $R = \bigoplus_{n=0}^{\infty} R_n$ over a field $R_0=K$ with $R=K[R_1]$.

\begin{Definition}
\label{integral}
{\em
A poset $P$ is called {\em integral} if there exists a noetherian graded domain $R = \bigoplus_{n=0}^{\infty} R_n$ over a field $R_0=K$ for which $R$ is an ASL on $P$ over $K$ with an injection $\varphi: P \hookrightarrow R_1$.  In particular, $R$ is a homogeneous domain.
}
\end{Definition}

Clearly, if a poset $P$ is integral, then $P$ has a unique minimal element.  Every ASL domain $R$ with $\dim R\leq 3$ is Cohen--Macaulay \cite{HW}.  In \cite{Terai}, an ASL domain with $\dim R= 4$ which is not Cohen--Macaulay is constructed.  Every distributive lattice is integral \cite{Hibi}.

\begin{Conjecture}[\cite{Hibi}]
\label{disprove}
Every modular lattice is integral.
\end{Conjecture}

\section{Hilbert functions of Cohen--Macaulay homogeneous domains}
Let $R = \bigoplus_{n=0}^{\infty} R_n$ be a homogeneous algebra over a field $R_0=K$ and $H(R,n)=\dim_K R_n$ its Hilbet function.  The Hilbert series $F(R,\lambda) = \sum_{n=0}^{\infty} H(R,n) \lambda^n$ of $R$ is of the form $(h_0 + h_1\lambda + \cdots + h_s \lambda^s)/(1 - \lambda)^d$, where each $h_i$ is an integer with $h_s \neq 0$ and where $d= \dim R$.  We call $h(R) = (h_0,h_1, \ldots, h_s)$ the {\em $h$-vector} of $R$.  If $R$ is Cohen--Macaulay, then each $h_i > 0$.  If $R$ is Gorenstein, then $h_i = h_{s-i}$ for all $1\leq i \leq s$.  We refer the reader to \cite{Sta78} for the detailed information about Hilbert functions of homogeneous algebras.

\begin{Lemma}[Stanley \cite{Sta91}]
    \label{Stanley}
    The $h$-vector $h(R) = (h_0,h_1, \ldots, h_s)$ of a Cohen--Macaulay homogeneous domain $R$ satisfies the inequalities
    \begin{eqnarray}
        \label{MIT}
    h_0+h_1 + \cdots + h_j \leq h_s + h_{s-1} + \cdots + h_{s-j}, \quad 1 \leq j \leq [s/2].
    \end{eqnarray}
\end{Lemma}

Let $R = \bigoplus_{n=0}^{\infty} R_n$ be a homogeneous ASL on a poset $P$ over a field $R_0=K$ with an injection $\varphi: P \hookrightarrow R_1$ and $\Delta(P)$ the order complex of $P$.  It follows from (ASL-1) that the $h$-vector of $R$ is equal to the nonzero components of the $h$-vector $h(\Delta(P))$ of $\Delta(P)$.  In other words, if $h(\Delta(P))=(h_0,h_1, \ldots, h_s, 0,\ldots,0)$ with $h_s \neq 0$, then $h(R) = (h_0,h_1,\ldots, h_s)$.  We refer the reader to \cite{StaGREEN} and \cite{HIBIred} for the background on combinatorics of simplicial complexes and their $h$-vectors.

\section{Non-integral modular lattices}
Let $L$ be a modular lattice and $R = \bigoplus_{n=0}^{\infty} R_n$ an ASL domain on $L$ over a field $R_0=K$ with an injection $\varphi: P \hookrightarrow R_1$.  Since every modular lattice is Cohen--Macaulay \cite{Bjo}, it follows from \cite[Corollary 4.2]{Eis} that $R$ is Cohen--Macaulay.  Hence, Lemma \ref{Stanley} guarantees that the $h$-vector $$h(\Delta(L))=(h_0,h_1, \ldots, h_s, 0, \ldots, 0), \quad h_s \neq 0$$ of $\Delta(L)$ satisfies the inequalities (\ref{MIT}).

Now, in order to disprove Conjecture \ref{disprove}, our task is to find a modular lattice $L$ for which $h(\Delta(L))$ does not satisfy the inequalities (\ref{MIT}).  Recall that a lattice $L$ is modular if $a,b,c \in L$ with $a \leq c$, then $a \vee(b \wedge c) = (a \vee b) \wedge c$.

Let $P$ be a poset and $x \in P$.  Let $P_x = P \cup x'$, where $x' \not\in P$, and define the partial order on $P_x$ as follows:

(i) if $b,c \in P_x$ with $b \neq x'$ and $c \neq x'$, then $b < c$ in $P_x$ if and only if $b < c$ in $P$;

(ii) if $b \in P_x$ with $b \neq x'$, then $b < x'$ (resp. $x' < b$) in $P_x$ if and only if $b < x$ (resp. $x < b$) in $P$;

(iii) $x$ and $x'$ are incomparable in $P_x$.

\noindent
We say that $P_x$ is a {\em duplication} of $P$ at $x$.

An element $x$ of a lattice is called {\em join-irreducible} (resp. {\em meet-irreducible}) if there is a unique element $y \in L$ for which $y < x$ (resp. $x < y$) and $y < z < x$ (resp. $x < z < y$) for no $z \in L$.

\begin{Lemma}
    \label{duplication}
Let $L$ be a modular lattice and suppose that $x \in L$ is join-irreducuble and meet-irreducible.  Then the duplication $L_x$ of $L$ at $x$ is a modular lattice.
\end{Lemma}

\begin{proof}
Since $x$ is join-irreducuble and meet-irreducible, if $a \neq b$, then $a \vee b \neq x$ and $a \wedge b \neq x$ for all $a,b \in L$.  Hence $L_x$ is a lattice.

A classical result by Dedekind \cite[Chapter 3, Exercise 30]{EC1} guarantees that a lattice $L$ is modular if and only if no sublattice of $L$ is the lattice $D_5$ of Figure 1.

\begin{figure}
\centering
\begin{tikzpicture}[scale=1.5]
\node[draw,shape=circle] (x0) at (0,1) {};
\node[draw,shape=circle] (x1) at (1,0) {};
\node[draw,shape=circle] (x2) at (2,0.5) {};
\node[draw,shape=circle] (x3) at (2,1.5) {};
\node[draw,shape=circle] (x4) at (1,2) {};
\draw(x0)--(x1)--(x2)--(x3)--(x4)--(x0);
\end{tikzpicture}
\caption{The lattice $D_5$.}
\end{figure}
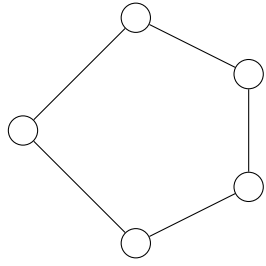

\noindent
Since no sublattice of $L$ is $D_5$ it follows that no sublattice of $L_x$ is $D_5$.  Hence $L_x$ is modular, as desired.
\, \, \, \, \, \, \, \, \, \, \, \, \, \, \, \, \, \, \, \, \, \, \, \, \, \, \, \, \, \, \, \, \, \, \, \, \, \, \, \, \,
 \end{proof}

\begin{Lemma}
    \label{exactly}
Let $P$ be a pure poset \cite[p.~115]{HIBIred} and $x \in P$.  Suppose that $x$ belongs to exactly one maximal chain of $P$.  Then
\[
h(\Delta(P_x)) = h(\Delta(P)) + (0,1,0,\ldots, 0).
\]
\end{Lemma}

\begin{proof}
Let $d - 1 = \dim \Delta(P)$ and $f_i$ the number of chains of $L$ of length $i$.  Let $h(\Delta(P)) = (h_0,h_1, \ldots, h_d)$.  One has
\[
\sum_{i=0}^{d} f_i (x-1)^{d-i} = \sum_{i=0}^{d} h_i x ^{d-i}.
\]
Let $h(\Delta(P_x)) = (h'_0,h'_1, \ldots, h'_d)$.  One has
\[
\sum_{i=0}^{d} f_i (x-1)^{d-i} + \sum_{i=1}^{d} { d-1 \choose i-1} (x-1)^{d-i} = \sum_{i=0}^{d} h'_i x ^{d-i}.
\]
Since
\begin{eqnarray*}
\sum_{i=1}^{d} { d-1 \choose i-1} (x-1)^{d-i}
& = & \sum_{i=0}^{d-1} { d-1 \choose i} (x-1)^{d-(i+1)} \\
& = & \sum_{i=0}^{d-1} { d-1 \choose i} (x-1)^{(d-1) - i} \\
& = & x^{d-1},
\end{eqnarray*}
the desired result follows.
\, \, \, \, \, \, \, \, \, \, \, \, \, \, \, \, \, \, \, \, \, \, \, \, \, \, \, \, \, \, \, \, \, \, \, \, \,
\end{proof}

\begin{Theorem}
    \label{Non-integral}
    A non-integral modular lattice exists.
\end{Theorem}

\begin{proof}
    Let $L$ denote the divisor lattice \cite[p.~157]{HHgtm260} of $2^s \cdot 3^t$ with $3 \leq s \leq t$.  One has $h(\Delta(L))= (h_0, h_1, \ldots, h_s, 0, \ldots, 0)$ with $h_s \neq 0$.  Now, $x=2^s \in L$ belongs to exactly one maximal chain of $L$. We define $L^{[n]} = (L^{[n-1]})_x$ for $n \geq 2$ with $L^{[1]} = L_x$.  Since $L$ is a distributive lattice, it follows from Lemmas \ref{duplication} and \ref{exactly} that $L^{[n]}$ is a modular lattice and
    \[
    h(\Delta(L^{[n]}) = (h_0, h_1+n, h_2, \ldots, h_s, 0, \ldots, 0).
    \]
    Hence for $n \gg 0$, $h(\Delta(L^{[n]})$ fails to satisfy the inequality (\ref{MIT}) and no homogeneous ASL on $L^{[n]}$ over $K$ can be an integral domain.
    \, \, \, \, \, \, \, \, \, \, \, \, \, \, \, \, \, \, \,
\end{proof}

\begin{figure}
\centering
\begin{tikzpicture}[scale=1.5]
\begin{scope}[rotate=45]
\node[draw,shape=circle] (x0) at (0,0) {};
\node[draw,shape=circle] (x1) at (0,1) {};
\node[draw,shape=circle] (x2) at (0,2) {};
\node[draw,shape=circle] (x3) at (0,3) {};
\node[draw,shape=circle] (x4) at (1,0) {};
\node[draw,shape=circle] (x5) at (1,1) {};
\node[draw,shape=circle] (x6) at (1,2) {};
\node[draw,shape=circle] (x7) at (1,3) {};
\node[draw,shape=circle] (x8) at (2,0) {};
\node[draw,shape=circle] (x9) at (2,1) {};
\node[draw,shape=circle] (x10) at (2,2) {};
\node[draw,shape=circle] (x11) at (2,3) {};
\node[draw,shape=circle] (x12) at (3,0) {};
\node[draw,shape=circle] (x13) at (3,1) {};
\node[draw,shape=circle] (x14) at (3,2) {};
\node[draw,shape=circle] (x15) at (3,3) {};
\node[draw,shape=circle] (x16) at (4,0) {};
\node[draw,shape=circle] (x17) at (4,1) {};
\node[draw,shape=circle] (x18) at (4,2) {};
\node[draw,shape=circle] (x19) at (4,3) {};
\end{scope}
\draw(x3)node[below left=0.5mm]{\large{$2^3$}}--(x2)node[below left=0.5mm]{\large{$2^2$}}--(x1)node[below left=0.5mm]{\large{$2$}}--(x0)--(x4)node[below right=0.5mm]{\large{$3$}}--(x8)node[below right=0.5mm]{\large{$3^2$}}--(x12)node[below right=0.5mm]{\large{$3^3$}}--(x16)node[below right=0.5mm]{\large{$3^4$}}--(x17)--(x18)--(x19)--(x15)--(x11)--(x7)--(x3);
\draw(x7)--(x6)--(x5)--(x4);
\draw(x11)--(x10)--(x9)--(x8);
\draw(x15)--(x14)--(x13)--(x12);
\draw(x2)--(x6)--(x10)--(x14)--(x18);
\draw(x1)--(x5)--(x9)--(x13)--(x17);
\end{tikzpicture}
\caption{The divisor lattice of $2^3\cdot 3^4$.}
\end{figure}
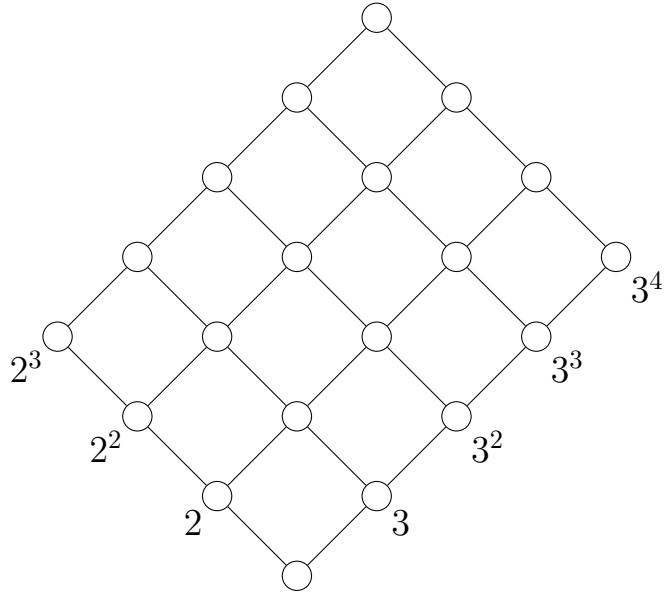

\section*{Acknowledgments}
The second author is supported by a FAPA grant from Universidad de los Andes.

\section*{Statements and Declarations}
The authors have no Conflict of interest to declare that are relevant to the content of this article.

\section*{Data availability}
Data sharing does not apply to this article as no new data were created or analyzed in this study.

\end{document}